\documentclass[preprint,12pt]{elsarticle}

\usepackage{amssymb}
\usepackage{amsthm}
\usepackage[all]{xy}
\newtheorem{thm}{Theorem}

\newtheorem{lemma}{Lemma}

\newcommand{\ff}{{\mathfrak f}}

\newcommand{\cF}{{\mathcal F}}

\newcommand{\Q}{{\mathbb Q}}
\newcommand{\R}{{\mathbb R}}
\newcommand{\Z}{{\mathbb Z}}
\newcommand{\N}{{\mathbb N}}
\newcommand{\C}{{\mathbb C}}

\begin{document}

\begin{frontmatter}



\title{Every totally real algebraic integer is a tree eigenvalue}


\author{Justin Salez}
\ead[url]{http://www.proba.jussieu.fr/$\sim$salez/}
\address{Universit\'e Paris Diderot \& LPMA}

\begin{abstract}
Graph eigenvalues are examples of totally real algebraic integers, i.e. roots of real-rooted monic polynomials with integer coefficients. Conversely, the fact that every totally real algebraic integer occurs as an eigenvalue of some finite graph is a deep result, conjectured forty years ago by Hoffman, and proved seventeen years later by Estes. This short paper provides an independent and elementary proof of a stronger statement, namely that the graph may actually be chosen to be a tree. As a by-product, our result implies that the atoms of the limiting spectrum of $n\times n$ symmetric matrices with independent Bernoulli$\,\left(\frac{c}{n}\right)$ entries ($c>0$ is fixed as $n\to\infty$) are exactly the totally real algebraic integers. This settles an open problem raised by Ben Arous (2010). 
\end{abstract}


\end{frontmatter}


\section{Introduction}
\label{}
By definition, the \emph{eigenvalues} of a finite graph $G=(V,E)$ are the roots of its characteristic polynomial $\Phi_G(x):=\det(xI-A)$, where $A=\{A_{i,j}\}_{i,j\in V}$ is the adjacency matrix of $G$:
$$A_{i,j}=
\left\{
\begin{array}{cl}
1 & \textrm{if } \{i,j\}\in E \\
0 & \textrm{otherwise}.
\end{array}
\right.
$$
Those eigenvalues capture a considerable amount of information about $G$. For a detailed account, see e.g.  \cite{cds,brouwer}. It follows directly from this definition that any graph eigenvalue is a \emph{totally real algebraic integer}, i.e. a root of some real-rooted monic polynomial with integer coefficients. Remarkably enough, the converse is also true: every totally real algebraic integer is an eigenvalue of some finite graph. This deep result was conjectured forty years ago by Hoffman \cite{hoffman}, and established seventeen years later by Estes \cite{estes}, see also \cite{bass}. The present paper provides an independent, elementary proof of a stronger statement, namely that the graph may actually be chosen to be a tree.
\begin{thm}
\label{th}
Every totally real algebraic integer is an eigenvalue of some finite tree.  
\end{thm}

Trees undoubtedly play a special role in many aspects of graph theory. We therefore believe that the strengthening provided by Theorem \ref{th} may be of independent interest, beyond the fact that it provides a simpler proof of Hoffman's conjecture. In addition, Theorem \ref{th} settles an open problem raised by Ben Arous \cite[Problem 14]{AIM}, namely that of determining the set  $\Sigma$ of atoms of the limiting spectral distribution of $n\times n$ symmetric matrices with independent Bernoulli$\,\left(\frac{c}{n}\right)$ entries ($c>0$ is fixed and $n\to\infty$). Indeed, it follows from the log-H\"older continuity of the spectrum of integer matrices at algebraic numbers (see e.g. \cite[Section 6]{LC}) that $\Sigma$ is contained in the set of totally real algebraic integers. On the other hand, $\Sigma$ is easily seen to contain all tree eigenvalues, as noted by Ben Arous. Theorem \ref{th} precisely states that those inner and outer bounds coincide, thereby settling the question.

\section{Outline of the proof}

Let $T$ be a finite tree with a distinguished vertex $o$ (the root). Removing $o$ naturally yields a decomposition of $T$ into smaller rooted trees $T_1,\ldots,T_d$ ($d\in\N$) as depicted in the following  diagram:
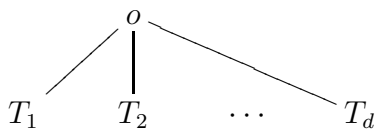
\begin{displaymath}
    \xymatrix{
         & o \ar@{-}[dl] \ar@{-}[d]  \ar@{-}[drr]  & \\
         T_1 & T_2 & \cdots & T_d}
\end{displaymath}
To such a rooted tree, let us associate the rational function 
\begin{eqnarray}
\label{f}
{\ff}_{T}(x) & = & 1-\frac{\Phi_T(x)}{x\Phi_{T\setminus o}(x)}\ = \ 1-\frac{\Phi_T(x)}{x\Phi_{T_1}(x)\cdots \Phi_{T_d}(x)}.
\end{eqnarray}
Expressed in terms of this function, the classical recursion for the characteristic polynomial of trees (see e.g. \cite[Proposition 5.1.1]{brouwer}) simply reads
\begin{eqnarray}
\label{rec}
{\ff}_{T}(x) & = & \frac{1}{x^2}\sum_{i=1}^d\frac{1}{1-\ff_{T_i}(x)},
\end{eqnarray}
the sum being interpreted as $0$ when empty (i.e. when $T$ is reduced to $o$). 
Now fix $\lambda\in\C\setminus\{0\}$. In view of (\ref{f}), the problem of finding a (minimal) tree with eigenvalue $\lambda$ is equivalent to that of finding $T$ such that $\ff_T(\lambda)=1$. Thanks to the recursion (\ref{rec}), this task boils down to that of generating the number $1$ from the initial seed $0$ using repeated applications of the maps
\begin{eqnarray*}
\left(\alpha_1,\ldots,\alpha_d\right) & \longmapsto & \frac{1}{\lambda^2}\sum_{i=1}^d\frac{1}{1-\alpha_i}\qquad (d\in\N).
\end{eqnarray*}
As an example, consider the golden ratio $\lambda=\frac{1+\sqrt{5}}{2}$. Iterating $4$ times the map $\alpha\mapsto\frac{1}{\lambda^2(1-\alpha)}$ and using the identity $\lambda^2=\lambda+1$ yields successively:
\begin{eqnarray*}
\begin{array}{ccccccc}
0 & \longrightarrow & \frac{1}{\lambda^2} 
  & \longrightarrow & \frac{1}{\lambda^2-1}=\frac{1}{\lambda}
  & \longrightarrow & \frac{1}{\lambda^2-\lambda}=1,
\end{array}
\end{eqnarray*}
which shows that $\lambda$ is an eigenvalue of the linear tree $T=P_4$. Remarkably enough, this seemingly specific argument can be extended in a systematic way to any totally real algebraic integer $\lambda$, and the numbers that may be produced in this way can be completely determined. To formalize this, let us fix a totally real algebraic integer $\lambda\neq 0$, and introduce the rational function
\begin{eqnarray}
\label{map}
\Psi(x) & = & \frac{1}{\lambda^2(1-x)}.
\end{eqnarray}
Define $\cF\subseteq \R$ as the smallest set containing $0$ and satisfying for all $d\geq 1$,
\begin{eqnarray}
\label{maps}
\alpha_1,\ldots,\alpha_d\in \cF\setminus\{1\} & \Longrightarrow & \Psi(\alpha_1)+\cdots+\Psi(\alpha_d)\in\cF.
\end{eqnarray}
We will prove that $\cF$ is nothing but the field generated by $\lambda^2$:
\begin{eqnarray}
\label{F}\cF=\left\{\frac{P(\lambda^2)}{Q(\lambda^2)}\colon P,Q\in\Z[x], Q(\lambda^2)\neq 0\right\}=:\Q(\lambda^2).
\end{eqnarray}
In particular, $1\in \cF$, and Theorem \ref{th} follows. The remainder of the paper is devoted to the proof of (\ref{F}). The detailed argument is given in Section \ref{proof}, while Section \ref{prelim} provides the necessary background on algebraic numbers.

\section{Algebraic preliminaries}
\label{prelim}
A number $\zeta\in\C$ is \emph{algebraic} if there is $P\in\Q[x]$ such that $P(\zeta)=0$. In that case, such  $P$  are exactly the multiples of a unique monic polynomial $P_0\in\Q[x]$, called the \emph{minimal polynomial} of $\zeta$. The algebraic number $\zeta$ is 
\begin{itemize}
\item \emph{totally real} if all the complex roots of $P_0$ are real ;
\item \emph{totally positive} if all the complex roots of $P_0$ are real and positive.
\end{itemize}
The following Lemma gathers the basic properties of algebraic numbers that will be used in the sequel. These are well-known (see e.g. \cite{lang}) and follow directly from the fact that if $P(x)=\prod_{i}(x-\alpha_i)$ and $Q(x)=\prod_{j}(x-\beta_j)$ have rational coefficients, then so do the polynomials
\begin{eqnarray*}
\prod_{i,j}(x-\alpha_i-\beta_j), \qquad & \displaystyle{\prod_{i}\left(x-\frac{1}{\alpha_i}\right),} & \qquad  \prod_{i,j}(x-\alpha_i\beta_j)\\
\prod_{i}(x+\alpha_i), \qquad & \displaystyle{\prod_{i}(x^2-\alpha^2_i),}& \qquad \prod_{i}(x^2-\alpha_i).
\end{eqnarray*}

\begin{lemma}[Elementary algebraic properties]\label{lm:algebra} $\qquad\qquad\qquad$
\renewcommand{\theenumi}{\alph{enumi}}
\begin{enumerate}
\item The totally real algebraic numbers form a sub-field of $\R$.
\item The set of totally positive algebraic numbers is stable under $+,\times,\div$.
\item If $\alpha$ is totally real and $\beta$ is totally positive, then $\alpha+n\beta$ is totally positive for all sufficiently large $n\in\N$. 
\item If $\alpha\neq 0$ is totally real, then $\alpha^2$ is totally positive.
\end{enumerate}
\end{lemma}

We shall also use twice the following elementary result.
\begin{lemma}
\label{lm:decomposition}
Let $\zeta$ be algebraic with minimal polynomial $P$. Set $n=\deg P$. Then for any $q_1>\ldots>q_n\in\Q$,  there exist $m_1,\ldots,m_n\in\Z$ such that
\begin{eqnarray}
\label{eq:decomposition}
\frac{m_1}{\zeta-q_1}+\cdots+\frac{m_n}{\zeta-q_n} & \in & \N^+=\{1,2,\ldots\},
\end{eqnarray}
and $m_k$ has the same sign as $(-1)^kP(q_k)$ for every $k\in\{1,\ldots,n\}$.  
\end{lemma}
\begin{proof}
For  $1\leq k\leq n$, consider the rational number
\begin{eqnarray*}
r_k  &:=&  \frac{-P(q_k)}{\prod_{j\neq k}(q_k-q_j)}.
\end{eqnarray*}
Note that $r_k$ has the same sign as $(-1)^kP(q_k)$, since $q_1>\ldots>q_n$. Moreover, 
\begin{eqnarray*}
r_1\prod_{k\neq 1}(x-q_k)+\cdots+r_n\prod_{k\neq n}(x-q_k) & = & \prod_{k=1}^n(x-q_k)-P(x).
\end{eqnarray*}
Indeed, both sides are polynomials of degree less than $n$, and they coincide at the $n$ points $q_1,\ldots,q_n$. 
Evaluating at $x=\zeta$ gives 
\begin{eqnarray*}
\frac{r_1}{\zeta-q_1}+\cdots+\frac{r_n}{\zeta-q_n} & =  & 1,
\end{eqnarray*}
and multiplying this by a large enough integer yields the result. 
\end{proof}

\section{Proof}
\label{proof}
Before we start, let us make three simple observations which will be used several times in the sequel. First, by (\ref{maps}), we have for any $k\in \N$,
\begin{eqnarray}
\label{fact}
\frac{1}{\lambda^2-k} & = & \Psi\left(\underbrace{\Psi\left(0\right)+\cdots+\Psi(0)}_{\textrm{k  terms}}\right)\in\cF.
\end{eqnarray}
Second, $\cF$ is stable under internal addition:
\begin{eqnarray}
\label{add}
\alpha,\beta\in\cF & \Longrightarrow & \alpha+\beta\in\cF.
\end{eqnarray} 
Indeed, the conclusion is trivial if $\alpha=0$ or $\beta=0$. Now if $\alpha,\beta$ are non-zero elements of $\cF$, then by construction they are of the form
\begin{eqnarray*}
\alpha = \Psi(\alpha_1)+\cdots+\Psi(\alpha_n) & \textrm{ and } &
\beta = \Psi(\beta_1)+\cdots+\Psi(\beta_m),
\end{eqnarray*}
for some $n,m\geq 1$ and $\alpha_1,\ldots,\alpha_n,\beta_1,\ldots,\beta_m$ in $\cF\setminus\{1\}$. But then,
\begin{eqnarray*}
\alpha+\beta & = &  \Psi(\alpha_1)+\cdots+\Psi(\alpha_n)+\Psi(\beta_1)+\cdots+\Psi(\beta_m),
\end{eqnarray*}
which, by (\ref{maps}), shows that $\alpha+\beta\in\cF$.  
Third, the field $\Q(\lambda^2)$ obviously contains $0$ and satisfies (\ref{maps}). Since $\cF$ is minimal with this property, we deduce that 
\begin{eqnarray}
\label{half}
\cF\subseteq\Q(\lambda^2).
\end{eqnarray} 
By part (a) of Lemma \ref{lm:decomposition}, we also get that all elements in $\cF$ are totally real.

\section*{Step 1: $\cF$ contains a positive integer}
We may assume that $1\notin\cF$, otherwise there is nothing to prove. Consequently, we do not need to worry about divisions by zero when applying $\Psi$ to an element $\alpha\in\cF$.  Let us first apply Lemma \ref{lm:decomposition} to $\zeta=\lambda^2$ with $q_k=k, 1\leq k \leq \deg\zeta$. From (\ref{fact}) and (\ref{add}), it follows that the sum appearing in (\ref{eq:decomposition}) is the difference of two elements in $\cF$. 
In other words, we have found $\Delta\in\N^+$ and $\alpha$ with the following property: 
\begin{eqnarray}
\label{shift}
\alpha\in \cF & \textrm{ and } & \alpha-\Delta\in\cF.
\end{eqnarray}
As already noted, $\alpha$ is totally real. In fact we may even assume that 
$1-\alpha$ is totally positive, because $\alpha'=\alpha+\beta$ also satisfies (\ref{shift}) for any $\beta\in\cF$, and choosing $\beta=\frac{j}{\lambda^2-k}$ with $j,k\in\N$ large enough eventually makes $1-\alpha'$ totally positive, by parts (b) and (c) of Lemma \ref{lm:algebra}. In turn, parts (b) and (d) now guarantee that
\begin{eqnarray}
\label{tp}
\xi:=\lambda^2(1-\alpha) & \textrm{ is totally positive}.
\end{eqnarray}

Now fix $j,k\in\N$ and set $i=(\Delta-1)j+1$. Since $\cF$ is stable under addition, property (\ref{shift}) and the fact that $\frac{1}{\lambda^2}=\Psi(0)\in\cF$ imply that
\begin{eqnarray*}
\underbrace{\alpha+\cdots+\alpha}_{i\textrm{ terms}}
+
\underbrace{(\alpha-\Delta)+\cdots+(\alpha-\Delta)}_{j\textrm{ terms}}
+
\underbrace{\frac{1}{\lambda^2}+\cdots+\frac{1}{\lambda^2}}_{k\textrm{ terms}} & \in & \cF.
\end{eqnarray*}
But this number equals $1-(\Delta j+1)(1-\alpha)+\frac{k}{\lambda^2}$, so applying $\Psi$ gives
\begin{eqnarray*}
\frac{1}{(\Delta j+1)\xi-k}  & \in  & \cF.  
\end{eqnarray*}
Adding up $\Delta j+1$ copies of this last number, we finally arrive at
\begin{eqnarray}
\label{dense}
\frac{1}{\xi-q}\in \cF & \textrm{ for any }& q\in{\cal Q}:=\left\{\frac{k}{\Delta j+1}:j,k\in\N\right\}.
\end{eqnarray}

We may now conclude: by (\ref{tp}), the minimal polynomial $P$  of $\xi$ has $n:=\deg P$ pairwise distinct positive roots. Since $\cal Q$ is dense in $[0,\infty)$, one can find $q_1>\cdots>q_n$ in $\cal Q$ that interleave those roots, in the sense that  
$P(q_k)$ has sign $(-1)^k$ for every $1\leq k \leq n$. Consequently, Lemma \ref{lm:decomposition} provides us with {non-negative} integers $m_1,\ldots,m_n$ such that 
$$
\frac{m_1}{\xi-q_1}+\cdots+\frac{m_n}{\xi-q_n}\in\N^+.
$$
On the other hand, this sum is in $\cF$ by (\ref{dense}) and (\ref{add}). Thus, $\cF\cap\N^+\neq\emptyset$.

\section*{Step 2: $(\cF,+)$ is a group}
 We know that $\cF$ contains some $d\in\N^+$. Since $d+d$ also belongs to $\cF$, we may assume without loss of generality that $d\neq 1$ to avoid divisions by zero below. Now fix $\alpha\in-\cF$ with $\alpha\neq 1$. Since $\cF$ is stable under addition,
$$d+\underbrace{(-\alpha)+\cdots+(-\alpha)}_{d-1\textrm{ terms}}\in \cF.$$
Applying $\Psi$ shows that $\frac{-1}{\lambda^2(1-\alpha)(d-1)}\in\cF$. Adding up $(d-1)$ copies of this number, we conclude that $\frac{-1}{\lambda^2(1-\alpha)}\in\cF$. We have  proved: 
\begin{eqnarray*}
\label{negation}
\alpha\in(-\cF)\setminus\{1\} & \Longrightarrow & \Psi(\alpha)\in -\cF.
\end{eqnarray*}
In other words, $-\cF$ is stable under $\Psi$. In view of (\ref{add}), we deduce that $-\cF$ satisfies (\ref{maps}). By minimality of $\cF$, we conclude that $\cF\subseteq-\cF$, i.e. that $\cF$ is stable under negation. Thus, the monoid $(\cF,+)$ is a group. 

\section*{Step 3: $\cF$ is the field $\Q(\lambda^2)$.}
In view of (\ref{half}), we only need to show that for $P,Q\in\Z[x]$ with $Q(\lambda^2)\neq 0$,
\begin{eqnarray}
\label{field}
\frac{P(\lambda^2)}{Q(\lambda^2)} & \in &\cF.
\end{eqnarray}
Since $\lambda^2$ is an algebraic integer, we may assume that $Q$ is monic with  $\deg Q>\deg P$ (otherwise, replace $Q$  with $Q+P_0^{\deg P}$, where $P_0$ denotes the minimal polynomial of $\lambda^2$). 
Let us prove the claim by induction over $n=\deg Q$. The case $n=0$ is simply the fact that $0\in\cF$. Now, assume that the claim holds for a certain $n\in\N$, and consider 
\begin{eqnarray*}
Q(x) & = & x^{n+1}+a_nx^n+\cdots+a_0,
\end{eqnarray*}
with $a_0,\ldots,a_n\in\Z$.  Let us first prove (\ref{field})  in the following two special cases:
\begin{itemize}
\item Case 1: $P(x)=x^n$. By our induction hypothesis, $\frac{1}{1+\lambda^{2n}}\in\cF$ and hence
$$\frac{1}{\lambda^{2n+2}}+\frac{1}{\lambda^2}=\Psi\left(\frac{1}{1+\lambda^{2n}}\right)\in\cF.$$
But $\cF$ also contains $\frac{1}{\lambda^2},\ldots,\frac{1}{\lambda^{2n}}$ by our induction hypothesis. Since $(\cF,+)$ is a group, one deduces that $\cF$ contains
$$-\left(\frac{a_n}{\lambda^2}+\cdots+\frac{a_0}{\lambda^{2n+2}}\right)=1-\frac{Q(\lambda^2)}{\lambda^{2n+2}}.$$ 
Finally, applying $\Psi$ shows that 
$\frac{\lambda^{2n}}{Q(\lambda^2)}=\frac{P(\lambda^2)}{Q(\lambda^2)}\in\cF$, as desired.
\item Case 2: $P$ is monic of degree $n$ with $P(0)=1$. Then
$$R(x):=P(x)-\frac{Q(x)-Q(0)P(x)}{x}$$
 is a polynomial over $\Z$ with $\deg R<n$. Thus, our induction hypothesis guarantees that $\cF$ contains $\frac{R(\lambda^2)}{P(\lambda^2)}$, hence $\frac{R(\lambda^2)}{P(\lambda^2)}-\frac{Q(0)}{\lambda^2}$ and hence also 
 $$\Psi\left(\frac{R(\lambda^2)}{P(\lambda^2)}-\frac{Q(0)}{\lambda^2}\right)=\frac{P(\lambda^2)}{Q(\lambda^2)}.$$
\end{itemize}
For the general case, note that every monomial $x^k$ $(0\leq k\leq n)$ may be written as a signed sum of polynomials $P$ of the form considered in the two special cases above. Since $(\cF,+)$ is a group, we conclude that (\ref{field}) holds for all $P\in\Z[x]$ with $\deg P\leq n$, and the induction is complete.

\section{Acknowledgment}
The author warmly thanks Arnab Sen for pointing out Problem 14 in \cite{AIM}.




\bibliographystyle{elsarticle-num}
\bibliography{draft}







\end{document}